\documentclass[11pt]{amsart}
\usepackage{amsmath, amsthm}
\usepackage[english]{babel}
\usepackage[T1]{fontenc}
\usepackage[latin1]{inputenc}
\usepackage{amssymb}
\usepackage{amscd}
\usepackage{latexsym}
\usepackage[bookmarksnumbered,plainpages,hypertex]{hyperref}
\usepackage{graphicx}
\usepackage{mathrsfs} 

\newtheoremstyle{theorem}
  {12pt}          
  {12pt}  
  {\sl}  
  {\parindent}     
  {\bf}  
  {. }    
  { }    
  {}     
\theoremstyle{theorem}
\newtheorem{theorem}{Theorem}

\newtheorem{remark}[theorem]{Remark}
\newtheorem{proposition}[theorem]{Proposition}
\newtheorem{lemma}[theorem]{Lemma}

\newtheorem{definition}[theorem]{Definition}

\linethickness{0.5pt}

\newcommand{\ic}{\ensuremath{\mathcal{I}}}

\newcommand{\oc}{\ensuremath{\mathcal{O}}}

\newcommand{\fc}{\ensuremath{\mathcal{F}}}

\newcommand{\Pt}{\mathbb{P}^3}
\newcommand{\PtD}{\mathbb{P}_3^*}

\newcommand{\Pn}{\mathbb{P}^n}

\newcommand{\oG}{\omega}

\newcommand{\fG}{\varphi}

\newcommand{\lag}{\langle}
\newcommand{\rag}{\rangle}

\newcommand{\bds}{\begin{displaystyle}}
\newcommand{\eds}{\end{displaystyle}}

\title[A restriction theorem.]{A restriction theorem for stable rank two vector bundles on $\Pt$.}

\author{Philippe ELLIA - Laurent GRUSON}
\address{Dipartimento di Matematica, 35 via Machiavelli, 44100 Ferrara, Italy}
\email{phe@unife.it}
\address{L.M.V., Université de Versailles-St Quentin, 45 Av. des Etats-Unis, 78035 Versailles, France}
\email{Laurent.Gruson@uvsq.fr}


\begin{document}
\maketitle


\begin{abstract}Let $E$ be a normalized, rank two vector bundle on $\Pt$. Let $H$ be a general plane. If $E$ is stable with $c_2(E) \geq 4$, we show that $h^0(E_H(1))\leq 2+c_1$. It follows that $h^0(E(1)) \leq 2+c_1$. We also show that if $E$ is properly semi-stable and indecomposable, $h^0(E_H(1))=3$.
\end{abstract}

\section{Introduction.}

We work over an algebraically closed field of characteristic zero. Let $E$ denote a stable, normalized ($-1 \leq c_1(E)\leq 0$) rank two vector bundle on $\Pt$. By Barth's restriction theorem (\cite{Barth}) if $H$ is a general plane, then $E_H$ is stable (i.e. $h^0(E_H)=0$) except if $E$ is a null-correlation bundle ($c_1=0, c_2=1$). In this note we prove:

\begin{theorem}
\label{T-thm}
Let $E$ be a stable, normalized, rank two vector bundle on $\Pt$. Assume $c_2(E)\geq 4$. Let $H$ be a general plane, then:\\
(a) $h^0(E_H(1)) \leq 1$ if $c_1=-1$ and\\
(b) $h^0(E_H(1)) \leq 2$ if $c_1=0$.\\
In particular it follows that $h^0(E(1)) \leq 2+c_1$.
\end{theorem}

The idea of the proof is as follows: if the theorem is not true then every general plane contains a unique line, $L$, such that $E_L$ has splitting type $(r, -r+c_1)$, $r \geq c_2-1$. We call such a line a ''super-jumping line''. Then we show that these super jumping lines are all contained in a same plane, $H$. The plane $H$ is very unstable for $E$. Performing a reduction step with $H$, we get a contradiction.

We observe (Remark \ref{R-sharp}) that the assumptions (and conclusions) of the theorem are sharp.\\
For sake of completeness we show (Proposition \ref{P-sstable}) that if $E$ is properly semi-stable, indecomposable, then $h^0(E_H(1))=3$ for $H$ a general plane.

\section{Proof of the theorem.}

We need some definitions:

\begin{definition}
Let $E$ be a stable, normalized rank two vector bundle on $\Pt$. A plane $H$ is \emph{stable} if $E_H$ is stable; it is \emph{semi-stable} if $h^0(E_H)\neq 0$ but $h^0(E_H(-1))=0$. A plane is \emph{special} if $h^0(E_H(-m))\neq 0$ with $m > 1$.\\
A line is general if the splitting type of $E_L$ is $(0,c_1)$. A line $L$ is a \emph{super jumping line} (s.j.l.) if the splitting type of $E_L$ is $(r, -r+c_1)$, with $r \geq c_2-1$.
\end{definition}

\begin{lemma}
\label{L-stable special planes}
Let $E$ be a stable, normalized rank two vector on $\Pt$. Assume $c_2(E)\geq 4$ and $h^0(E_H(1)) > 2+c_1$ if $H$ is a general plane. Then:\\
(i) Every stable plane contains a unique s.j.l. all the other lines are general or, if $c_1=0$, of type $(1,-1)$.\\
(ii) A semi-stable plane contains at most one s.j.l.\\
(iii) There is at most one special plane.
\end{lemma}

\begin{proof} (i) If $H$ is a stable plane every section of $E_H(1)$ vanishes in codimesion two:
$$0 \to \oc _H \to E_H(1) \to \ic _{Z,H}(2+c_1) \to 0\,\,\,(*)$$
We have $h^0(\ic _{Z,H}(2+c_1)) \geq 2+c_1$. If $c_1=-1$, $Z$ has degree $c_2$ and is contained in a line $L_H$. If $c_1=0$, we have $h^0(\ic _{Z,H}(2)) \geq 2$. Since $\deg (Z)=c_2+1 > 4$, the conics have a fixed line, $L_H$, and there is left a pencil of lines to contain the residual scheme of $Z$ with respect to $L_H$. It follows that the residual scheme is one point and that $length\,(Z\cap L_H)=c_2$. So in both cases there is a line, $L_H$, containing a subscheme of $Z$ of length $c_2$. Restricting $(*)$ to $L_H$ we get $E_{L_H} \to \oc _{L_H}(1+c_1-c_2)$. It follows that the splitting type of $E_{L_H}$ is $(c_2-1,c_1-c_2+1)$, hence $L_H$ is a s.j.l. If $L\neq L_H$ is another line in $H$, let $s$ be the length of $L\cap Z$. Restricting $(*)$ to $L$ we get:
$$0 \to \oc _L(s-1) \to E_L \to \oc _L(c_1-s+1) \to 0$$
This sequence splits except maybe if $c_1=s=0$ (in this case the splitting type is $(0,0)$ or $(1,-1)$). If $L$ is a s.j.l. then $s \geq c_2$, hence $L = L_H$. This shows that a stable plane contains a unique s.j.l. Since $s \leq 1$ (resp. $s \leq 2$) if $c_1=-1$ (resp. $c_1=0$), a line different from $L_H$ is general or has splitting type $(1,-1)$.\\
(ii) If $H$ is semi-stable then we have:
$$0 \to \oc _H \to E_H \to \ic _{T,H}(c_1) \to 0\,\,\,(**)$$
Here $\deg (T)=c_2$. If $L$ is a line in $H$ let $s$ denote the length of $L\cap T$. From $(**)$ we get: $0 \to \oc _L(s) \to E_L \to \oc _L(c_1-s)\to 0$. This sequence splits, so the splitting type of $E_L$ is $(s, -s+c_1)$. If $L$ is a s.j.l. then $s \geq c_2-1$ and $L$ contains a subscheme of length at least $\deg (T)-1$ of $T$. Since $c_2 \geq 4$, such a s.j.l. is uniquely defined. This shows that an unstable plane contains at most one s.j.l.\\
(iii) We may assume $h^0(E_H(-m-1))=0$. We have:
$$0 \to \oc _H \to E_H(-m) \to \ic _{X,H}(c_1-2m) \to 0\,\,\,(***)$$
If $L$ is a general line of $H$ ($L \cap X=\emptyset$) then $E_L$ has splitting type $(m, -m+c_1)$, with $m > 1$. 

Let's show that such a special plane, if it exists, is unique. Assume $H_1, H_2$ are two special planes. Let $H$ be a general stable plane. If $L_i=H\cap H_i$, then $L_1, L_2$ are two lines of $H$ with splitting type $(k_i, -k_i+c_1)$, $k_i > 1$. By (i) this is impossible. 
\end{proof}

We are ready for the proof of the theorem.

\begin{proof}[Proof of Theorem \ref{T-thm}:]\quad \\
Let $U \subset \PtD$ be the dense open subset of stable planes. We have a map $\fG :U \to G(1,3)$ defined by $\fG (H)=L_H$ where $L_H$ is the unique s.j.l. contained in $H$. So $\fG$ gives a rational map $\fG : \PtD ---> G(1,3)$. We claim that $\fG$ doesn't extend as a morphism to $\PtD$. Indeed in the contrary case we would have a section of the incidence variety $I =\{(H,L)\mid L \subset H\} \to \PtD$. Since $I \simeq Proj(\Omega _{\PtD}(1))$ (indeed the fibre at $H$ of $\Omega _{\PtD}(1)$ is the hyperplane corresponding to $H$), such a section corresponds to an injective morphism of vector bundles $\oc _{\PtD} \hookrightarrow T_{\PtD}(k)$, for some $k$. But there is no twist of $T_{\PtD}$ with a non-vanishing section. This can be seen by looking at $c_3(T_{\PtD}(k))$ or with the folllowing argument: the quotient would be a rank two vector bundle with $H^1_* =0$, hence, by Horrocks' theorem, a direct sum of line bundles which is absurd.

If $H$ is a singular point of the ''true'' rational map $\fG$, then, by Zariski's Main Theorem, $H$ contains infinitely many s.j.l. This implies that $H$ is the unique special plane (and that $\fG$ has a single singular point). We claim that every s.j.l. is contained in $H$. Indeed let $R$ be a s.j.l. not contained in $H$. Let $z = R\cap H$. There exists a s.j.l. $L \subset H$ through $z$. The plane $\lag R, L\rag$ contains two s.j.l. hence it is special: contradiction.

Since there are $\infty ^2$ s.j.l. we conclude that the general splitting type on the special plane $H$ is $(c_2-1, -c_2+c_1+1)$. So $m = c_2-1$ i.e. $h^0(E_H(-c_2+1))\neq 0$ (and this is the least twist having a section). Now we perform a reduction step (see \cite{SRS} Prop. 9.1).

If $c_1=0$ we get:
$$0 \to E' \to E \to \ic _{W,H}(-c_2+1) \to 0$$
where $E'$ is a rank two reflexive sheaf with Chern classes $c'_1=-1, c'_2=1, c'_3=c_2^2-c_2+1$. Since $E$ is stable, $E'$ too is stable. By \cite{SRS} Theorem 8.2 we get a contradiction.

If $c_1 =-1$, since $E_H^* = E_H(1)$ we get:
$$0 \to E'(-1) \to E \to \ic _{R,H}(-c_2) \to 0$$
where the Chern classes of $E'$ are: $c'_1=0, c'_2=0$, $c'_3=c_2^2$. Since $E$ is stable $E'$ is semi-stable. By \cite{SRS} Theorem 8.2 we get, again, a contradiction.
\end{proof}

\begin{remark}
The argument to show that $\fG$ doesn't extend to a morphism is taken from \cite{Franco}. Another way to prove this is to consider the surfaces $S_L$: if $L$ is a general line every plane through $L$ is (semi-)stable, the general one being stable. So almost every plane of the pencil contains a unique s.j.l. taking the closure yields a ruled surface $S_L$. Then one shows that $S_L \neq S_D$ if $L,D$ are general and then concludes by looking at $S_L \cap S_D$ (see \cite{Ellia-vanish}).
\end{remark}

\begin{remark}
\label{R-sharp}
The assumption $c_2 \geq 4$ cannot be weakened. If $c_1=-1$ every stable rank vector bundle, $E$, with $c_1=-1, c_2=2$ is such that $h^0(E_H(1))=2$ for a general plane $H$ (see \cite{Ha-Sols}).
If $E(1)$ is associated to four skew lines, then $h^0(E_H(1))=3$ for $H$ general and $c_i(E)=(0,3)$.\\
On the other hand a special t'Hooft bundle, ($E(1)$ associated to $c_2+1$ disjoint lines on a quadric) is stable with $c_1(E)=0$ and, if $c_2\geq 4$, satisfies $h^0(E_H(1))=2$ for $H$ general.\\
By the way, Theorem \ref{T-thm} gives back $h^0(E(1))\leq 2$ for an instanton, a result first proved by Boehmer and Trautmann (see \cite{N-Trautmann} and references therein).\\ 
Finally let $E(1)$ be associated to the disjoint union of $c_2/2$ double lines of arithmetic genus -2. Then $E$ is stable with $c_1=-1$ and, if $c_2>2$, $h^0(E_H(1))=1$ for $H$ general.
\end{remark}

Concerning properly semi-stable bundles ($c_1(E)=0, h^0(E)\neq 0, h^0(E(-1))=0$) we have:

\begin{proposition}
\label{P-sstable}
Let $E$ be a properly semi-stable rank two vector bundle on $\Pt$. Assume $E$ indecomposable. If $H$ is a general plane then $h^0(E_H(1))=3$.
\end{proposition}

\begin{proof} We have $0 \to \oc \to E \to \ic _C\to 0$, where $C$ is a curve ($E$ doesn't split) with $\oG _C(4)\simeq \oc _C$. Twisting and restricting to a general plane: $0 \to \oc _H(1) \to E_H(1) \to \ic _{C\cap H,H}(1)\to 0$. If $h^0(\ic _{C\cap H,H}(1))\neq 0$ it follows from a theorem of Strano (\cite{Strano}, \cite{Ellia-lacunes}) that $C$ is a plane curve, but this is impossible ($\oG _C(4) \not \simeq \oc _C$ for a plane curve).
\end{proof}

\begin{remark}
\label{R-sstable}
To apply Strano's theorem we need $ch(k)=0$ (see \cite{Ha-Fe}). The previous argument gives a quick proof of Theorem \ref{T-thm} in case $c_1=-1, h^0(E(1))\neq 0$. In fact this remark has been the starting point of this note.
\end{remark}  

\begin{remark}
\label{R-reflexive}
Let $C$ be a plane curve of degree $d$. A non-zero section of $\oG _C(3)\simeq \oc _C(d)$ yields: $0 \to \oc \to \fc (1) \to \ic _C(1) \to 0$, where $\fc$ is a stable rank two reflexive sheaf with Chern classes $(-1, d, d^2)$. If $H$ is a general plane, $h^0(\fc _H(1))= 2$ if $d > 1$ (resp. $3$ if $d=1$). Similarly, considering the disjoint union of a plane curve and of a line, we get stable reflexive sheaf with $c_1(\fc )=0$ and $h^0(\fc _H(1))=3$. So Theorem \ref{T-thm} doesn't hold for stable reflexive sheaves. The interested reader can try to classify the exceptions.
\end{remark}



\end{document}